\documentclass[11pt,twoside, a4paper, english, reqno]{amsart}
\usepackage[dvips]{epsfig}
\usepackage{amscd}
\usepackage{amssymb}
\usepackage{amsthm}
\usepackage{amsmath}
\usepackage{latexsym}
\usepackage{amsaddr}

\usepackage{upref}
\usepackage{hyperref}
\usepackage{bookmark}

\hypersetup{linkcolor=blue, colorlinks=true
,citecolor = red}
\usepackage{color}
\newtheorem{thm}{Theorem}[section]

\newtheorem{lem}[thm]{Lemma}

\newtheorem{prop}[thm]{Proposition}

\newtheorem{rem}[thm]{Remark}

\newcommand{\Lq}{L^{q}(\Omega)}

\newcommand{\Li}{L^{\infty}(\Omega)}
\newcommand{\Om}{\Omega}
\newcommand{\RO}{R_{\Omega}}
\newcommand{\rn}{\mathbb{R}^{n}}

\newcommand{\CO}{C^{\infty}_c(\Omega\setminus\{0\})}
\newcommand{\Wnz}{W^{1,n}_{0}(\Omega)}

\newcommand{\Inb}{I_{n}[u,B_{1},R]}

\newcommand{\Inu}{I_{n}[u,\Om,R]}
\newcommand{\Jnv}{J_{n}[v,\Om,R]}

\newcommand{\Rt}{\tilde{R}}
\newcommand{\Br}{B(x,\tilde{R})}

\newcommand{\RNum}[1]{\lowercase\expandafter{\romannumeral #1\relax}}
\newcommand{\authorfootnotes}{\renewcommand\thefootnote{\@fnsymbol\c@footnote}}%

\numberwithin{equation}{section} \allowdisplaybreaks

\begin{document}

\title{An Improved Leray-Trudinger Inequality}
\author{Arka Mallick $^\dagger$ and Cyril Tintarev$^\diamond$}
\thanks{$^\dagger$ TIFR  Centre for Applicable Mathematics, Post Bag No. 6503
 Sharadanagar,\\Yelahanka New Town, Bangalore 560065.\\ Email:arka@math.tifrbng.res.in\\
 $^\diamond$Email:tintarev@math.uu.se}
\begin{abstract}
 In this article, we have derived the following Leray-Trudinger type inequality on a bounded domain $\Om$ in $\rn $ containing the origin.
\begin{align*}
\displaystyle{\sup_{u\in W^{1,n}_{0}(\Om), \Inu\leq 1}}\int_{\Omega} e^{c_n\left(\frac{|u(x)|}{E_{2}^{\beta}(\frac{|x|}{R})}\right)^{\frac{n}{n-1}}} dx < +\infty \ \text{, for some } c_n>0 \ \text{depending only on } n.  
\end{align*}
Here $\beta = \frac{2}{n}$, 
$I_n[u,\Om,R] := \int_{\Om}|\nabla u |^{n}dx- \left(\frac{n-1}{n}\right)^{n}\int_{\Om}\frac{|u|^{n}}{|x|^{n}E_{1}^n(\frac{|x|}{R})}dx $, $R \geq \displaystyle{\sup_{x\in \Om}}|x|$ and $E_{1}(t) := \log(\frac{e}{t})$, $E_{2}(t) := \log(eE_1(t))$ for $t\in (0,1].$ This improves an earlier result by Psaradakis and  Spector. Also we have proved that for any $c>0$ in the place of $c_n$ the above inequality is false if we take $\beta < \frac{1}{n}.$ 
\end{abstract}
\maketitle
\section{Introduction}
In this article we intend to discuss some Leray-Trudinger type inequalities. These type of inequalities are closely related to different types of Trudinger-Moser inequalities and Hardy inequalities. So let us quickly recall some relevant results about these inequalities. \\
\par Let $\Om$ be any bounded domain in $\rn.$ Then the celebrated Sobolev embedding theorem asserts that for $p<n$
\begin{align*}
W^{1,p}_{0}(\Om) \subset L^{q}(\Om) \ \ \text{for any } q \ \text{satisfying } 1\leq q \leq p^{*} = \frac{np}{n-p}.
\end{align*}
This immediately leaves us with the following question. What happens when $p=n?$ Note that here $p^{*}$ is formally $+\infty.$ As expected, in this case one can prove that $W^{1,n}_{0}(\Om)\subset L^{q}(\Om)$ for any $q$ satisfying, $1\leq q<+\infty.$ However, for $q= +\infty$, easy example shows that one can not get the above embedding. This raises the the following interesting question. What is the maximal growth function $f(t)$ for which the following is true:
\begin{align}\label{ILTII1}
 u\in W^{1,n}_{0}(\Om) \ \text{implies } \ \int_{\Om}f(u)dx <\infty?
\end{align}
By Sobolev embedding, \eqref{ILTII1} is true when $f$ is a polynomial, but in fact it is allowed to have exponential growth, namely there exists some constant $c(n)>0$ depending only on $n$ such that
\begin{align}\label{ILTII2}
\displaystyle{\sup_{u\in W^{1,n}_{0}(\Om),\left|\left|\nabla u\right|\right|_{L^{n}(\Om)}\leq 1}} \int_{\Om} e^{c(n)\left|u\right|^{\frac{n}{n-1}}}dx <\infty.
\end{align}
This result is most often attributed to Trudinger \cite{T}, although it was obtained earlier by Yudovich \cite{Yud}, Peetre \cite{Pet} and Pohozaev \cite{Poh}. The proof in this paper still follows the method of \cite{T} based on getting some expedient upper bound on the $L^{q}$ norm of $u$ that yields \eqref{ILTII2}.
Later Hempel, Morris and Trudinger (see \cite{HMT}) together proved that the function $f(t)= e^{t^{\frac{n}{n-1}}}$ is in fact the function with maximal growth. In 1971, Moser proved a sharp version of \eqref{ILTII2} in \cite{M}. In fact, he proved that for any $c$ satisfying, $0<c\leq \alpha(n)$,  \eqref{ILTII2} holds true but fails for any $c > \alpha(n),$ where $\alpha(n) = nw_{n-1}^{\frac{1}{n-1}}$ and $w_{n-1}$ is as usual surface area  of the unit ball in $\rn.$ This lead us to a very rich literature. See \cite{AD},\cite{AS1},\cite{AY},\cite{BM1},\cite{C},\cite{LR},\cite{P},\cite{R},\cite{T1} and references therein for more information.\\
\par Now let us discuss some results concerning Hardy inequalities. Let $B_{n}$ be the unit ball in $\rn$. Then the Hardy inequality says 
\begin{align}\label{BHI1}
H_n(u): = \int_{B_n} |\nabla u|^2 dx - \int_{B_n} \frac{|u|^2}{(1-|x|^{2})^2} dx \geq 0 \ ,\ \forall u \in H^1_0(B_n). 
\end{align}
When $n\geq 2$ one can prove that (see \cite{BM}) there exists a positive constant $C$, depending only on $n$ such that  
\begin{align}\label{BHI2}
H_n(u) \geq C\int_{B_n} u^2 dx ,\ \ \forall u \in H^1_0(B_n). 
\end{align}
Later Maz'ya proved in \cite{MV} that for $n>2$ there exists a constant $C_{n}>0$, depending only on $n$ such that for any $u \in H^1_0(B_n),$
\begin{align}\label{BHI3}
H_n(u) \geq C_{n} \left(\int_{B_n} |u(x)|^{\frac{2n}{n-2}}\right)^{\frac{n-2}{n}}.
\end{align}
This inequality is called Hardy-Sobolev-Maz'ya inequality. This inequality raises the following question. If $n=2$ wheather one can derive the Trudinger-Moser type of inequality or not. In \cite{GD} Wang and Ye proved that indeed in this case, one can derive such an inequality. Their result is the following.
\par \textit{
There exists an constant $C_0> 0$ such that
\begin{align}\label{BHI4}
\int_{B_2} e^{\frac{4\pi u^2}{H_2(u)}} dx \leq C_0 < \infty, \ \ \forall u \in H^1_0(B_2)\setminus \{0\}.
\end{align}
}
One may wonder if this kind of Moser-Tridunger inequality holds for bounded convex domain with smooth boundary. For that let us recall the corresponding version of boundary Hardy inequality. Let $\Om $ be a bounded, convex domain in $\mathbb{R}^2$ with smooth boundary, then (see \cite{BM}) 
\begin{align}\label{BHI5}
H_d(u) : = \int_{\Om} |\nabla u|^2 dx - \frac{1}{4}\int_{\Om} \frac{u^2}{d(x,\partial \Om)^2} dx>0 , \ \ \forall u \in H^1_0(\Om)\setminus \{0\}.
\end{align}
One can easily check that for $\Om = B_2$, \eqref{BHI4} is true if we replace $H_2(u)$ by $H_d(u).$  In fact, in a very recent paper  by Lu and Yang (\cite{LY}), it has been established that the above type of inequality is true for any bounded and convex domain. 
In this context let us mention the following improved version of Moser-Trudinger inequality on unit disk on $\mathbb{R}^2$ which was proved in \cite{AT,MS}.
\begin{align}\label{BHI6}
\displaystyle{\sup_{u\in H^1_0(B_2),||\nabla u||_{L^2(B_2)}\leq 1}}\int_{B_2} \frac{e^{4\pi u^2}-1}{(1-|x|^2)^2} dx < \infty .
\end{align}  
 \par Now one may ask the following question. What will be the scenario if we replace the distant function from boundary by simply the distant function from the origin in \eqref{BHI5}? To give an answer of this question, let us recall the corresponding Hardy inequality and the literature associated with it.  
 \par Let $\Om$ be a domain domain in $\rn$ containing the origin. Then the classical Hardy's inequality asserts that for $n\geq 3,$ 
\begin{align}\label{ILTII3}
I_{\Om}[u]:= \int_{\Om}|\nabla u|^{2}dx - \left(\frac{n-2}{2}\right)^{2}\int_{\Om}\frac{|u|^{2}}{|x|^{2}}dx \geq 0 \ \ , \forall u \in H^{1}_{0}(\Om).
\end{align}
Here $\left(\frac{n-2}{2}\right)^{2}$ is the best constant and never achieved. Hence there is a scope of improvement of \eqref{ILTII3}. The first improvement was done by Brezis and Vazquez (see \cite{BV}). Their result is the following.
\par \textit{If Lebesgue measure of $\Om$ is finite then, for any $1\leq q < 2^{*}$ there exists a positive constant $C$ depending only on $n$ , $q$ and $\Om$, such that
\begin{align}\label{ILTII4}
\left(I_{\Om}[u]\right)^{\frac{1}{2}} \geq C\left(\int_{\Om}|u|^{q}\right)^{\frac{1}{q}}\ \ , \forall u\in H^{1}_{0}(\Om).
\end{align}
Here as usual $2^{*}= \frac{2n}{n-2}$ is the critical Sobolev constant.} \\ \par In the past few years a lot of efforts have been made to improve  \eqref{ILTII4} and generalize it to $L^{p}$ Hardy-Sobolev inequality.  We refer \cite{ACR}, \cite{ASk}, \cite{AS}, \cite{FT}, \cite{VZ} and references therein for further details. Observing the relevance of \eqref{ILTII4} with Sobolev's inequality one may wonder what would be the case when $n=2$ or in analogy with $L^{p}$ Hardy-Sobolev inequality what would be the case when $p=n.$ For that let us recall Hardy's inequality for $p=n$ (see \cite{ACR},\cite{AS}, \cite{BFT1} and \cite{BFT}).\\
\par \textit{Let $\Om$ be a bounded domain in $\rn$ ($n\geq 2$) containing the origin and $\RO:=\displaystyle{\sup_{x\in\Om}|x|}$. Then for any $u \in W^{1,n}_{0}(\Om)$ and $R\geq\RO$
\begin{align}
I_{n}[u,\Om,R]:= \int_{\Om}|\nabla u |^{n}dx- \left(\frac{n-1}{n}\right)^{n}\int_{\Om}\frac{|u|^{n}}{|x|^{n}E_{1}^n(\frac{|x|}{R})}dx \geq 0,
\end{align}\label{ILTII5}
where $E_{1}(t):= \log(\frac{e}{t})$ for $t\in(0,1]$ and $\left(\frac{n-1}{n}\right)^{n}$ is the best possible constant which is never achieved. 
}\\
Inequality of type \eqref{BHI4} is generally expected to hold when the functional $H_2$ is replaced by a similar nonnegative functional with a general potential,
\begin{equation}
\label{q2}
H_V(u)=\int_{\Omega}|\nabla u|^2 dx-\int_{\Omega}V(x)|u|^2 dx
\end{equation}
as long as $H_V$ does not possess a null sequence, that is, a sequence $u_k\in C^\infty_c(\Omega)$, such that $H_V(u_k)\to 0$, while for some bounded open set $B$, $\int_Bu_k dx=1$ (such sequence is known to converge to a positive solution, often called generalized ground state). If a null-sequence exists, the corresponding analog of \eqref{BHI4} is immediately false. On the other hand, it is shown in \cite{T1} for $\Omega=B_2$ and for radial potential $V$, that if $H_V$ does not admit a null-sequence, then the analog of \eqref{BHI4} remains true, as long as conditions of the ground state alternative ( \cite{PinTin}) are satisfied. In the radial case this amounts to the condition (2.3) of \cite{T1} (we bring attention to the reader that we correct here a misprint in the published version):
\begin{equation}
\label{t1c}
\exists \alpha>0:\;\lim_{r\to 0}r^2(\log\frac{1}{r})^{2+\alpha}V(r)\to 0.
\end{equation}
In the case when $H_V$ has the form $I_2$, potential $V$ corresponds to the borderline case $\alpha=0$, does not satisfy condition \eqref{t1c}, and, remarkably, as it was first observed by Psaradakis and Spector in \cite{PS}, the analog of \eqref{BHI4} is false, and, more generally, 
there does not exist any positive constant $c$ depending only on $n$ for which the following is true:
\begin{align*}
\displaystyle{\sup_{u\in W^{1,n}_{0}(\Om),I_n[u,\Om,\RO] \leq 1}}\int_{\Om} e^{c|u|^{\frac{n}{n-1}}}dx < \infty  .
\end{align*}
However, introducing a logarithmic factor, in the same paper they proved the following Hardy-Trudinger type inequality. 
\par \textit{
Let $\Om$ be a bounded domain in $\rn$; $n\geq 2,$ containing the origin. For any $\epsilon> 0$ there exist positive constants $A_{n,\epsilon}$ depending only on $n$, $\epsilon$ and $B_{n}$ depending only on $n$ such that 
\begin{align}\label{ILTII6}
\int_{\Om} e^{A_{n,\epsilon}\left(\frac{|u|}{E_{1}^{\epsilon}\left(\frac{|x|}{\RO}\right)}\right)^{\frac{n}{n-1}}}dx \leq B_{n}vol(\Om) \ \ \text{for all } u\in W^{1,n}_{0}(\Om) \ \text{satisfying } I_{n}[u,\Om,\RO] \leq 1.
\end{align}
}\\
\par Let $B$ denote the unit ball in $\rn$ centered at the origin and consider the following space of radial functions
\begin{align*}
W:= \{u\in C_{0,rad}^{1}(B):u \ \text{is radially symmetric and } I_n[u,B,1]<\infty\} .
\end{align*}
Then one can easily derive using Lemma 5 of \cite{CR} that there exists an positive constant $C_{n}$ depending only on $n$ such that
\begin{align}\label{MILTITP}
\displaystyle{\sup_{u\in W,I_{n}[u,B,1]\leq1}} \int_{B} e^{C_{n}\frac{|u|^{\frac{n}{n-1}}}{E_{2}(|x|)}}dx <\infty,
\end{align}
where $E_{2}(t):= \log(eE_{1}(t))$ for $t \in (0,1].$
\par This simple observation motivated us to investigate wheather \eqref{ILTII6} is true if we replace $E_{1}$ by $E_{2}$ in the power of exponential and we got  the following result.
\begin{thm}\label{MILTIT}
Let $\ 0\in \Omega$ be a bounded domain in $\rn,$ $\ (n \geq 2)$. Then for any $\ \beta \geq 2/n,$ and $R\geq \RO,$ 
there exists constants $\ A_{n}$  and $\ B_{n}$  depending only on $n$  such that, for any  $\ 0<c<A_{n}$
\begin{align}\label{MILTIT1}
&\int_{\Omega} e^{c\left(\frac{|u(x)|}{E_{2}^{\beta}(\frac{|x|}{R})}\right)^{\frac{n}{n-1}}} dx \leq B_{n} vol(\Omega),\ 
\forall u \in W^{1,n}_{0}(\Omega) \ \text{satisfying,}\ \Inu \leq 1.
\end{align}
In other words,
\begin{align}\label{MILTIT2}
 \displaystyle{\sup_{u\in W^{1,n}_{0}(\Om), \Inu\leq 1}}\int_{\Omega} e^{c\left(\frac{|u(x)|}{E_{2}^{\beta}(\frac{|x|}{R})}\right)^{\frac{n}{n-1}}} dx < +\infty \ \text{, for } 0<c<A_{n}.
\end{align}
Moreover, the above supremum is $+ \infty$ if $\beta<\frac{1}{n}$ for any $c>0.$
\end{thm}
\begin{rem}The situation is not clear when $\beta$ satisfies $\frac{1}{n}\leq \beta <\frac{2}{n}$. However, for $n=2$, when $\beta$ satisfies $\frac{1}{n}\leq \beta <\frac{2}{n}$, \eqref{MILTIT1} is true if we consider $\Om $ to the unit ball in $\rn$ and $u \in W $ defined above. This fact follows from \eqref{MILTITP} .
\end{rem}

\section{Preliminaries}
In this section we are going to recall some lemmas, propositions and theorems that we have used in the subsequent sections. But before that let us fix the notations for the rest of this article. 
\par $\Om$ is a bounded domain in $\rn$ and $\RO=\displaystyle{\sup_{x\in \Om}|x|}.$  We define $E_{1}(t) := \log(\frac{e}{t})$ , $E_{2}(t):= \log(eE_{1}(t))$ for $t\in (0,1]$ and $C_{1}(n):= 2^{n-1}-1.$
\par The first lemma is a very basic one. It is basically a representation formula for smooth functions in terms of its derivative. The proof of this lemma can be found in \cite{GT}  (Lemma 7.14).
\begin{lem}\label{Grep1}
Let $\Om$ be any open set in $\rn,$ ($n\geq 2$) and $u\in C_{c}^{1}(\Om)$. Then ,
\begin{align*}
u(x) = \frac{1}{nw_{n}}\int_{\Om}\frac{(x-y).\nabla u(y)}{|x-y|^{n}} dy,
\end{align*}
where, $w_{n}$ is the volume of unit ball in $\rn$. 
\end{lem}
 The proof of following result can be found in \cite{PS} (Proposition 2.6). Easy to see here that we will get an equality instead of the inequality for $n=2$. 
\begin{prop} \label{CHTI1}
Suppose, $u\in \CO,$ and $u\geq 0$.Then$\   \ \Jnv \leq C_1(n)\Inu$. Here $C_1(n)= 2^{n-1}-1$, $\ v(x) = E_{1}^{-\frac{n-1}{n}}(\frac{|x|}{R})u(x)\  \forall x \in \Om,$ $J_{n}[v,\Om,R] = \int_{\Om} |\nabla v(x)|^{n}E_{1}^{n-1}(\frac{|x|}{R})dx$ and $R \geq \RO.$ 
\end{prop}
The following theorem is due to Barbatis, Filippas and Tertikas(See \cite{BFT1} Theorem B and Proposition 3.2 ). We are going to state a much simplified version of their result, which is about an improved version of $L^{n}$ Hardy's inequality.
\begin{thm}\label{SE CrHI1}
Let $\Om$ be a bounded domain in $\rn$. Then for any $R\geq \RO$  and all  $u \in W^{1,n}_0(\Om)$ there holds

\begin{align}\label{SE CrHIT}
\int_{\Om} |\nabla  u(x)|^{n} dx- \left(\frac{n-1}{n} \right)^{n} \int_{\Om} \frac{|u|^{n}}{|x|^{n}E_{1}^{n}(\frac{|x|}{R})}dx \geq  \frac{1}{2} \left(\frac{n-1}{n} \right)^{n-1} \int_{\Om} \frac{|u|^{n}}{|x|^{n}E_{1}^{n}(\frac{|x|}{R})E_{2}^{2}(\frac{|x|}{R})}dx .
\end{align}
Moreover, if $\Om$ contains the origin and for $R\geq \RO$ there exists a positive constant $B>0$ for which the following holds true 
\begin{align*}
\int_{\Om} |\nabla  u(x)|^{n} dx- \left(\frac{n-1}{n} \right)^{n} \int_{\Om} \frac{|u|^{n}}{|x|^{n}E_{1}^{n}(\frac{|x|}{R})}dx \geq  B \int_{\Om} \frac{|u|^{n}}{|x|^{n}E_{1}^{n}(\frac{|x|}{R})E_{2}^{\gamma}(\frac{|x|}{R})}dx,\ \forall u \in W^{1,n}_0(\Om)
\end{align*}
 and for some $\gamma \in \mathbb{R}$, then 
\begin{itemize}
\item $\gamma \geq 2$ ;
\item and if $\gamma = 2$ then $B \leq \frac{1}{2} \left(\frac{n-1}{n} \right)^{n-1}$.
\end{itemize}
This makes $\frac{1}{2} \left(\frac{n-1}{n} \right)^{n-1}$  the best constant of inequality \eqref{SE CrHIT}.
\end{thm}
For the rest of this article we will denote $\int_{\Om} |\nabla  u(x)|^{n} dx- \left(\frac{n-1}{n} \right)^{n} \int_{\Om} \frac{|u|^{n}}{|x|^{n}E_{1}^{n}(\frac{|x|}{R})}dx$ as $\Inu$.

\vspace{5mm}
\section{Estimate of $L^{q}$ norm}
The proposition that we are going to prove in this section will help us to prove the first part of our main theorem. Our plan is to follow Trudinger's approach of proving Trudinger inequality. To execute this, we need an expedient upper bound of $L^{q}$ norm of $u/ E_2^{2/n}$. In the next proposition we will derive such an upper bound. 

\begin{prop}\label{ILqes}
Let $u \in \Wnz,$ then for any $R \geq \RO$  and $q>n,$ we have the following estimate,
\begin{align}\label{ILqes1}
\left(\int_{\Om}\left|\frac{u(x)}{E_{2}^{\frac{2}{n}}(\frac{|x|}{R})}\right|^{q} dx\right)^{1/q}\leq C_{n} 
 \left[1+\frac{q(n-1)}{n}\right]^{1-\frac{1}{n}+\frac{1}{q}}
\left(vol(\Omega)\right)^{\frac{1}{q}}\left(\Inu\right)^{1/n}.
\end{align}
  Where 
$C_{n} = \frac{1}{n w_{n}^{\frac{1}{n}}} \left[\left(C_{1}(n)\right)^{\frac{1}{n}}+2^{\frac{1}{n}}\left(\frac{n}{n-1}\right)^{\frac{n-1}{n}}\frac{n+1}{n}\right] $
\end{prop}
\begin{proof}
 First of all note that, it is enough to prove the proposition for positive valued smooth functions. So, let $u\in \CO$, $u \geq 0$ and define, $v(x) = E_{1}^{-\frac{n-1}{n}}(\frac{|x|}{R})u(x)$, $\forall x\in \Om.$ Then using Lemma \ref{Grep1} and the fact $E_{1}(\frac{|x|}{R}),E_{2}(\frac{|x|}{R}) \geq 1$, $\forall x \in \Om,$ we have for  $x\in \Om,$
\begin{align}\label{ILqes2}
\left|\frac{u(x)}{E_{2}^{2/n}(\frac{|x|}{R})}\right| &=\left| \frac{v(x)E_{1}^{{1-\frac{1}{n}}}(\frac{|x|}{R})}{{\left( E_{2}(\frac{|x|}{R})\right)^{2/n}}}\right|\notag \\
&=\left| \frac{1}{nw_{n}} \int_{\Omega} \frac{(x-y). \nabla \left(v(y)\frac{E_{1}^{1-\frac{1}{n}(\frac{|y|}{R})}}{\left(E_{2}(\frac{|y|}{R})\right)^{2/n}}\right)}{|x-y|^{n}} dy \right| 
\notag \\
&\leq \frac{1}{n w_{n}}\int_{\Omega}\frac{|\nabla v(y)| E_{1}^{\frac{n-1}{n}}(\frac{|y|}{R})}{|x-y|^{n-1}} dy +
\frac{1}{n w_{n}} \int_{\Omega} \frac{v(y)}{|x-y|^{n-1}}\left |\nabla\left( \frac{E_{1}^{1-\frac{1}{n}}(\frac{|y|}{R})}{(E_{2}(\frac{|y|}{R}))^{2/n}}\right)\right| dy
\end{align}
Now,
\begin{align*}
\left| \nabla \left(\frac{E_{1}^{1-\frac{1}{n}}}{(E_{2})^{2/n}}\right)(\frac{|y|}{R})\right| 
&= \left|\frac{-E_{1}^{-\frac{1}{n}}(\frac{|y|}{R})(1-\frac{1}{n})\frac{y}{|y|^{2}} E_{2}^{2/n}(\frac{|y|}{R}) + \frac{y}{|y|^2}\frac{2}{n} E_{2}^{\frac{2}{n} -1}(\frac{|y|}{R}) E_{1}^{-\frac{1}{n}}(\frac{|y|}{R})}{E_{2}^{\frac{4}{n}}(\frac{|y|}{R})}\right|\\
&= \left|-\frac{y}{|y|^2 E_{1}^{\frac{1}{n}}(\frac{|y|}{R})}  \  \frac{(\frac{n-1}{n})E_{2}(\frac{|y|}{R}) - \frac{2}{n}}{ E_{2}^{\frac{2}{n}+1}(\frac{|y|}{R})}\right|\\
&\leq \frac{\frac{n+1}{n}}{|y|E_{1}^{\frac{1}{n}}(\frac{|y|}{R})E_{2}^{\frac{2}{n}}(\frac{|y|}{R})}  \ \ \ \ \  \ (\text{Since,}\   E_{2} \geq 1, \ \text{on} \ \Omega)
\end{align*}
So, \eqref{ILqes2} implies,
\begin{align*}
\left|\frac{u(x)}{E_{2}^{\frac{2}{n}}(\frac{|x|}{R})}\right| &\leq \frac{1}{n w_{n}} \int_{\Omega}\frac{|\nabla v(y)|E_{1}^{1-\frac{1}{n}}(\frac{|y|}{R})}{|x-y|^{n-1}} dy 
\\&\hspace{2.3cm}+\frac{n+1}{n^{2} w_{n}}\int_{\Omega} \frac{v(y)}{|y|E_{1}^{\frac{1}{n}}(\frac{|y|}{R})E_{2}^{\frac{2}{n}}(\frac{|y|}{R})|x-y|^{n-1}} dy . 
\end{align*}
Define,
\begin{align*}
K(x) &:=  \int_{\Omega}\frac{|\nabla v(y)|E_{1}^{1-\frac{1}{n}}(\frac{|y|}{R})}{|x-y|^{n-1}} dy \\
M(x) &:=  \int_{\Omega} \frac{v(y)}{|y|E_{1}^{\frac{1}{n}}(\frac{|y|}{R}) E_{2}^{\frac{2}{n}}(\frac{|y|}{R})|x-y|^{n-1}} dy .
\end{align*}
So, we have,
\begin{align} \label{ILqes3}
\left|\left|\frac{u}{E_{2}^{2/n}}\right|\right|_{\Lq} \leq \frac{1}{n w_{n}}\left[\left|\left|K\right|\right|_{\Lq} + \frac{n+1}{n}\left|\left|M\right|\right|_{\Lq}\right]
\end{align}
Now, let $q>n.$ Define $r$ by $1/n + 1/r = 1+1/q.$ Then clearly, $1<r<\frac{n}{n-1}.$ For $x \in \Omega,$  let us define
\begin{align*}
h_{r}(x) := \int_{\Omega} \frac{1}{|x-y|^{(n-1)r}} dy
\end{align*}
Let $\tilde{R}$ be such that $vol(\Om)= w_{n}\Rt^{n}.$ For $x\in \Om $, let us denote $\Br$ to be the ball of radius $\Rt$ and centre at $x$ in $\rn.$ Then as $vol(\Om) = vol(\Br) ,$ we have $vol\left(\Om \cap \Br^{c}\right) = vol\left(\Om^{c}\cap \Br\right).$ Now,
\begin{align*}
h_{r}(x) &= \int_{\Om \cap \Br^{c}} \frac{1}{|x-y|^{(n-1)r}} dy + \int_{\Om \cap \Br} \frac{1}{|x-y|^{(n-1)r}} dy\\
&\leq \frac{1}{\Rt^{(n-1)r}} vol\left(\Om \cap \Br^{c}\right)+ \int_{\Br \cap \Om}\frac{1}{|x-y|^{(n-1)r}} dy\\
&= \frac{vol\left(\Om^{c}\cap \Br\right)}{\Rt^{(n-1)r}}+\int_{\Br \cap \Om} \frac{1}{|x-y|^{(n-1)r}} dy \\
&= \int_{\Om^{c}\cap\Br}\frac{1}{\Rt^{(n-1)r}} dy + \int_{\Om \cap \Br} \frac{1}{|x-y|^{(n-1)r}} dy\\
&\leq \int_{\Om^{c}\cap\Br}\frac{1}{|x-y|^{(n-1)r}} dy+ \int_{\Om \cap \Br}\frac{1}{|x-y|^{(n-1)r}} dy\\
&= \int_{\Br}\frac{1}{|x-y|^{(n-1)r}} dy \\ 
&= \frac{n w_{n}\Rt^{n-(n-1)r}}{n-(n-1)r}.
\end{align*}
Using $1/n + 1/r = 1 + 1/q $ we get
\begin{align}\label{ILqes4}
\left|\left|h_{r}\right|\right|^{\frac{1}{r}}_{\Li} &\leq w_{n}^{1-\frac{1}{n}+\frac{1}{q}}\left(1+\frac{q(n-1)}{n}\right)^{1-\frac{1}{n}+\frac{1}{q}} \Rt^{\frac{n}{q}} \notag \\
&\leq w_n^{1-\frac{1}{n}} \left(1+\frac{q(n-1)}{n}\right)^{1-\frac{1}{n}+\frac{1}{q}} \left(vol(\Omega)\right)^{\frac{1}{q}} \ \ (\text{Since,}\ vol(\Omega) = w_{n} \Rt^{n}).
\end{align}
Now let us break the integrand of $K(x)$ as
\begin{align*}
\frac{|\nabla v(y)| E_{1}^{1-\frac{1}{n}}(\frac{|y|}{R})}{|x-y|^{n-1}} &= \left(\frac{|\nabla v(y)|^{n}E_{1}^{n-1}(\frac{|y|}{R})}{|x-y|^{(n-1)r}}\right)
^{\frac{1}{q}} \frac{1}{|x-y|^{(n-1)(1-\frac{r}{q})}} \left(|\nabla v(y)|^{n}E_{1}^{n-1}(\frac{|y|}{R})\right)^{\frac{1}{n}-\frac{1}{q}}.
\end{align*}
Note that, $\ \frac{1}{q}+ \frac{n-1}{n}+ \frac{1}{n} - \frac{1}{q} = 1.$ 
So, by applying H{\"o}lder's inequality with exponent $q,$ $\ \frac{n}{n-1}$  and $\  \frac{1}{\frac{1}{n}-\frac{1}{q}}$ we get,
\begin{align*}
K(x) &\leq  \left(\int_{\Omega} \frac{|\nabla v(y)|^{n}E_{1}^{n-1}(\frac{|y|}{R})}{|x-y|^{(n-1)r}}dy\right)^
{\frac{1}{q}} \left(h_{r}(x)\right)^{1-\frac{1}{n}}\left(\int_{\Omega}|\nabla v(y)|^{n}E_{1}^{n-1}(\frac{|y|}{R})dy\right)^{\frac{1}{n}-\frac{1}{q}}.
\end{align*}
Integrating,
\begin{align*}
 ||K||_{\Lq} &\leq \left|\left|h_{r}\right|\right|_{\Li}^{1-\frac{1}{n}} \left(\int_{\Omega} |\nabla v|^{n}E_{1}^{n-1}(\frac{|y|}{R})dy\right)^{\frac{1}{n}-\frac{1}{q}}\left(\int_{\Om}\int_{\Om} \frac{|\nabla v(y)|^{n}E_{1}^{n-1}(\frac{|y|}{R})}{|x-y|^{(n-1)r}} dydx\right)^{1/q}.
\end{align*}
So, by Tonelli's theorem, we have,
\begin{align*}
||K||_{\Lq} &\leq \left|\left|h_{r}\right|\right|_{\Li}^{1-\frac{1}{n}} \left(\int_{\Omega} |\nabla v|^{n}E_{1}^{n-1}(\frac{|y|}{R})dy\right)^{\frac{1}{n}-\frac{1}{q}}\left(\int_{\Om}|\nabla v(y)|^{n}E_{1}^{n-1}(\frac{|y|}{R})h_{r}(y) dy\right)^{1/q}.
\end{align*}
Now we use Proposition \ref{CHTI1} to get,
\begin{align}\label{ILqes5}
||K||_{\Lq} \leq &\leq \left(C_{1}(n)\right)^{\frac{1}{n}} \left|\left| h_{r}\right|\right|_{\Li}^{\frac{1}{r}}\left(\Inu\right)^{\frac{1}{n}} .
\end{align}
Similarly writing integrand of $M(x)$ as
\begin{align*}
\frac{|v(y)|}{|y|E_{1}^{\frac{1}{n}}(\frac{|y|}{R}) E_{2}^{\frac{2}{n}}(\frac{|y|}{R})} \frac{1}{|x-y|^{n-1}} &= 
\left(\frac{|v(y)|^{n}}{|y|^{n}E_{1}(\frac{|y|}{R}) E_{2}^{2}(\frac{|y|}{R})|x-y|^{(n-1)r}}\right)^{\frac{1}{q}}\\&\frac{1}{|x-y|^{(n-1)(1-\frac{r}{q})}} 
\left(\frac{|v(y)|^{n}}{|y|^{n}E_{1}(\frac{|y|}{R}) E^{2}_{1}(\frac{|y|}{R})}\right)^{\frac{1}{n}-\frac{1}{q}}
\end{align*}
and applying H{\"o}lder's inequality with the same exponent as in the case of $K(x)$ and Tonelli's theorem we get,
\begin{align*}
||M||_{\Lq} 
&\leq \left|\left| h_{r} \right |\right|_{\Li}^{\frac{1}{r}} \left(\int_{\Omega}\frac{v^{n}(y)}{|y|^{n}E_{1}(\frac{|y|}{R})E_{2}^{2}(\frac{|y|}{R})} dy\right)^{\frac{1}{n}} \\
&= \left|\left| h_{r} \right |\right|_{\Li}^{\frac{1}{r}} \left(\int_{\Omega}\frac{u^{n}(y)}{|y|^{n}E_{1}^{n}(\frac{|y|}{R})E_{2}^{2}(\frac{|y|}{R})} dy\right)^{\frac{1}{n}}.
\end{align*}
In the last inequality we have put $v(x) = E_{1}^{-\frac{n-1}{n}}(\frac{|x|}{R})u(x).$ Now we use Theorem \ref{SE CrHI1} of previous section to get,
\begin{align*}
\Big(\int_{\Omega}\frac{u^{n}(y)}{|y|^{n}E_{1}^{n}(\frac{|y|}{R})E_{2}^{2}(\frac{|y|}{R})}dy\Big)^{\frac{1}{n}} \leq 2^{\frac{1}{n}}
\big(\frac{n}{n-1}\big)^{\frac{n-1}{n}} \big(\Inu\big)^{\frac{1}{n}}.
\end{align*}
So,
\begin{align} \label{ILqes6}
||M||_{\Lq} \leq \left|\left| h_{r}\right|\right|_{\Li}^{\frac{1}{r}} \ 2^{\frac{1}{n}}
\left(\frac{n}{n-1}\right)^{\frac{n-1}{n}} \left(\Inu\right)^{\frac{1}{n}}.
\end{align}
Using \eqref{ILqes4} ,\eqref{ILqes5} and \eqref{ILqes6} we have from \eqref{ILqes2},
\begin{align*}
\left|\left|\frac{u}{E_{2}^{\frac{2}{n}}}\right|\right|_{\Lq} \leq 
\left[C_{1}^{\frac{1}{n}}(n)+2^{\frac{1}{n}}\left(\frac{n}{n-1}\right)^{\frac{n-1}{n}}\frac{n+1}{n}\right]
\frac{1}{nw_{n}^{\frac{1}{n}}}&\left[1+\frac{q(n-1)}{n}\right]^{1-\frac{1}{n}+\frac{1}{q}} \\ 
&\left(vol(\Omega)\right)^{\frac{1}{q}}\left(\Inu\right)^{\frac{1}{n}}.
\end{align*}
This proves \eqref{ILqes1}.
\end{proof}
\section{Proof of Theorem}
\begin{proof}
Let $\ u \in W^{1,n}_{0}(\Omega)$ satisfying $\ \Inu\leq 1.$ Then applying Proposition \ref{ILqes} 
with $\ q= \frac{nk}{n-1},$ $\ k\in\{n,n+1,.....\}$ we get,
\begin{align*}
\int_{\Omega}\left(\frac{|u(x)|}{E_{2}^{2/n}(\frac{|x|}{R})}\right)^{\frac{nk}{n-1}} dx \leq C_{n}^{\frac{nk}{n-1}} \ vol(\Omega) \ (1+k)^{1+k}.
\end{align*}
Multiplying above inequality by $\ \frac{c^{k}}{k!}$ and adding from $ n $ to$\ m$,($n\leq k \leq m$) we get,
\begin{align*}
\int_{\Omega} \sum_{k=n}^{m}\frac{1}{k!}\left(c\left(\frac{|u(x)|}{E_{2}^{2/n}(\frac{|x|}{R})}\right)^\frac{n}{n-1}\right)^{k} dx \leq 
\sum_{k=n}^{m}\left(c\ C_{n}^{\frac{n}{n-1}}\right)^{k}\ vol(\Omega) \ \frac{(1+k)^{1+k}}{k!}.
\end{align*}
R.H.S converges as $\ m\rightarrow \infty$ if $ c<\frac{1}{e\ C_{n}^{\frac{n}{n-1}}}.$ Also by H{\"o}lder's inequality 
\begin{align*}
S = \int_{\Omega} \sum_{s=0}^{n-1}\left(c \left(\frac{|u(x)|}{ E_{2}^{2/n}(\frac{|x|}{R})}\right)^{\frac{n}{n-1}}\right)^{s} dx,
\end{align*}
is bounded by a constant depending only on $n$. This proves the first part of the result for $\ \beta = 2/n$. For $\beta > 2/n,$ the same $ A_{n}$
 and $B_{n}$ will work since $\ E_{2}(\frac{|x|}{R}) \geq 1$, $\forall x\in \Om.$\\ 
 \par Next we will prove the second part of our result i.e. for $\beta <\frac{1}{n}$ the inequality \eqref{MILTIT1} is false. Let $\Om = B_{1},$ be the unit ball in $\rn$ centered at the origin and $\beta < \frac{1}{n}.$  If possible
 let $c_{1}$ and $c_{2}$ be two positive constants depending only on $n$ such that
\begin{align}\label{MILTI}
  \int_{B_{1}} e^{c_{1}\left(\frac{|u|}{E_{2}^{\beta}}\right)^{\frac{n}{n-1}}} dx < c_{2} \ , \forall u \in W^{1,n}_{0}(B_{1})\ \text{satisfying } \Inb \leq 1,
\end{align}
where $R\geq1$ is a fixed real number. We choose $1<\theta <2$ such that $1<n\beta+\theta<2.$ In the rest of the proof we will concentrate our focus on deriving the following inequality
\begin{align}\label{MILTI1}
 \int_{B_{1}} \frac{|u|^{n}(x)}{|x|^{n}E_{1}^{n}(\frac{|x|}{R})E_{2}^{n\beta+\theta}(\frac{|x|}{R})} dx \leq c \Inb \ , \forall u \in W^{1,n}_{0}(B_{1}),
\end{align}
for some constant $c>0.$ But  Theorem \ref{SE CrHI1} suggests that this can only happen if $n\beta+\theta \geq 2.$ Hence we are through. \\
 \par Let $u\in W^{1,n}_{0}(B_{1})$ be such that $I_{n}[u,B_{1},R] \leq 1. $ Then
 \begin{align}\label{MILTI2}
 \int_{B_{1}} &\frac{|u|^{n}}{|x|^{n}E_{1}^{n}(\frac{|x|}{R})E_{2}^{n\beta+\theta}(\frac{|x|}{R})} dx \notag \\ &= \left(\frac{n-1}{c_{1}}\right)^{n-1}
 \int_{B_{1}}\left[\left(\frac{c_{1}}{n-1}\right)^{n-1}\frac{|u|^{n}}{E_{2}^{n\beta}(\frac{|x|}{R})}\right]\left[\frac{1}{|x|^{n}E_{1}^{n}(\frac{|x|}{R})E_{2}^{\theta}(\frac{|x|}{R})}\right] dx\notag\\
 &\leq \left(\frac{n-1}{c_{1}}\right)^{n-1}2^{n-2}\int_{B_{1}} e^{c_{1}\left(\frac{|u|}{E_{2}^{\beta}(\frac{|x|}{R})}\right)^{\frac{n}{n-1}}} dx + P_{B_{1}},
\end{align}
where 
\begin{align*}
 P_{B_{1}} = 2^{2(n-2)} \int_{B_{1}} \left(1+\frac{1}{|x|^{n}E_{1}^{n}(\frac{|x|}{R})E_{2}^{\theta}(\frac{|x|}{R})}\right)\left(\log\left(1
                                   + \left(\frac{1}{|x|^{n}E_{1}^{n}(\frac{|x|}{R})E_{2}^{\theta}(\frac{|x|}{R})}\right)^{\frac{1}{n-1}}\right)\right)^{n-1} dx.
\end{align*}
To derive the last inequality of \eqref{MILTI2} we have used the following inequality
\begin{align*}
 ab \leq 2^{n-2}\left[e^{(n-1)a^{\frac{1}{n-1}}}+ 2^{n-2}(1+b)\left(\log(1+b^{\frac{1}{n-1}})\right)^{n-1}\right]\ , \forall a,b\geq 0,
\end{align*}
which in turn can be derived from the following version of Young's inequality 
\begin{align*}
 ab \leq e^{a}- a -1+(1+b)\log(1+b)-b \ \text{(See \cite{PS} for details)}.
\end{align*}
Now using \eqref{MILTI} we get from \eqref{MILTI2}
\begin{align}\label{MILTI3}
 \int_{B_{1}} \frac{|u|^{n}}{|x|^{n}E_{1}^(\frac{|x|}{R})E_{2}^{n\beta+\theta}(\frac{|x|}{R})} dx \leq c_{2}\left(\frac{n-1}{c_{1}}\right)^{n-1}2^{n-2}+ P_{B_{1}}.
\end{align}
An easy calculation shows that $P_{B_{1}}$ is finite (See \cite{PS}, Theorem 1.1 for details). Hence for general $u \in W_{0}^{1,n}(B_{1}),$ putting $\tilde u = \frac{u}{(\Inb)^{\frac{1}{n}}}$ in 
\eqref{MILTI3} we get \eqref{MILTI1}. This concludes the theorem.
\end{proof}
\section*{Acknowledgement}
The authors would like to thank Prof. Sandeep and Prof Adimurthi for their helpful suggestions.


\begin{thebibliography}{101}
\bibitem{ACR} Adimurthi; Chaudhuri, N.; Ramaswamy, M. : 
\emph{An improved Hardy-Sobolev inequality and its application.} Proc. Amer. Math. Soc. 130 (2002), no. 2, 489-505.
\bibitem{AD} Adimurthi; Druet, O. : 
\emph{Blow-up analysis in dimension 2 and a sharp form of Trudinger-Moser inequality.} Comm. Partial Differential Equations 29, no. 1-2, 295-322 (2004).

\bibitem{ASk}Adimurthi; Sekar, A.: 
\emph{Role of the fundamental solution in Hardy-Sobolev-type inequalities.} Proc. Roy. Soc. Edinburgh Sect. A 136 (2006), no. 6, 1111-1130.
\bibitem{AS1} Adimurthi; Sandeep, K. : 
\emph{A singular Moser-Trudinger embedding and its applications.} NoDEA Nonlinear Differ. Equ. Appl. 13(5-6), 585-603 (2007).
\bibitem{AS} Adimurthi; Sandeep, K.: 
\emph{Existence and non-existence of the first eigenvalue of the perturbed Hardy-Sobolev operator.} Proc. Roy. Soc. Edinburgh Sect. A 132 (2002), no. 5, 1021-1043.
\bibitem{AT} Adimurthi; Tintarev, C.: 
\emph{On a version of Trudinger?Moser inequality with Möbius shift invariance.} Calc. Var. Partial
Differential Equations 39 (2010) 203-212.
\bibitem{AY} Adimurthi ;  Yang, Y.  : 
\emph{An interpolation of hardy inequality and trudinger-moser inequality in $\rn$ and its applications.} International Mathematics Research Notices, vol. 13, 2394-2426 (2010).
\bibitem{BFT1}Barbatis, G.; Filippas, S.; Tertikas, A. :
\emph{Series expansion for $L^{p}$ Hardy inequalities.} 
Indiana Univ. Math. J. 52 (2003), no. 1, 171-190.
\bibitem{BFT}Barbatis, G.; Filippas, S.; Tertikas, A. :
\emph{A unified approach to improved $L^p$ Hardy inequalities with best constants.} 
Trans. Amer. Math. Soc. 356 (2004), no. 6, 2169-2196
\bibitem{BM1} Battaglia, L.; Mancini, G. : 
\emph{Remarks on the Moser-Trudinger Inequality.} Adv. Nonlinear Anal. 2, no. 4, 389-425 (2013).
\bibitem{BM}  Brezis, H.;  Marcus, M. : 
\emph{Hardy's inequality revisited.} Ann. Sc. Norm. Pisa 25 (1997) 217-237.
\bibitem{BV} Brezis, H.; Vazquez, J L. :
\emph{Blow-up solutions of some nonlinear elliptic problems.} Rev. Mat. Univ. Complut. Madrid 10 (1997), no. 2, 443-469.
\bibitem{C} Cao, D. : 
\emph{Nontrivial solution of semilinear elliptic equations with critical exponent in $\mathbb{R}^2$.} Communications in Partial Differential Equations, vol. 17, 407-435 (1992).
\bibitem{CR} Calanchi, M.; Ruf, B: 
\emph{Trudinger-Moser type inequalities with logarithmic weights in dimension N.} Nonlinear Anal. 121 (2015), 403-411.
\bibitem{FT}Filippas, S.; Tertikas, A. : 
\emph{Optimizing improved Hardy inequalities.} J. Funct. Anal. 192 (2002), no. 1, 186-233.
\bibitem{GT}Gilbarg,D; Trudinger, N :
\emph{Elliptic partial differential equations of second order.} 
Second edition. Grundlehren der Mathematischen Wissenschaften [Fundamental Principles of Mathematical Sciences], 224. Springer-Verlag, Berlin, 1983. xiii+513 pp. ISBN: 3-540-13025-X
35Jxx.
\bibitem{HMT} Hempel, J. A. ; Morris, G. R. ; Trudinger, Neil. S. :
\emph{ On the sharpness of a limiting case of the Sobolev imbedding theorem.} 
Bull. Austral. Math. Soc. 3 1970 369-373.
\bibitem{LR} Li, Y.;  Ruf, B. : 
\emph{A sharp Trudinger-Moser type inequality for unbounded domains in $\rn$.} Indiana University Mathematics Journal, vol. 57, no. 1, 451-480 (2008).
\bibitem{LY} Lu, G. ; Yang, Q. : 
\emph{A sharp Trudinger-Moser inequality on any bounded and convex planar domain} arXiv:1512.07163v1 (Submitted on 22 Dec 2015 ).
\bibitem{M} Moser,J. : 
\emph{ A sharp form of an inequality by N. Trudinger.} Indiana Univ. Math. J. 20 (11) (1971) 1077-1092.
\bibitem{MS} Mancini, G.; Sandeep, K. :
\emph{Moser-Trudinger inequality on conformal disks.} Commun. Contemp. Math. 12 (6) (2010)
1055-1068.

\bibitem{MV} Maz'ya, V.G.: 
\emph{ Sobolev Spaces.} Springer Verlag, Berlin, New York, 1985.
\bibitem{P} Panda, R. : 
\emph{Nontrivial solution of a quasilinear elliptic equation with critical growth in $\rn$.} Proceedings
of the Indian Academy of Science, vol. 105, pp. 425-444 (1995).
\bibitem{Pet} Peetre, J. : \emph{Espaces d'interpolation et theoreme de Soboleff, Ann. Inst.}
Fourier (Grenoble) 16 (1966), 279-317.
\bibitem{PinTin} Pinchover, Y. ; Tintarev, K. : 
\emph{Ground state alternative for
p-Laplacian with potential term.} 
Calc. Var. Partial Differential Equations
28 (2007), no. 2, 179-201.

\bibitem{Poh} Pohozhaev,S.I. : \emph{The Sobolev imbedding in the case pl = n.} Proc. Tech.
Sci. Conf. on Adv. Sci. Research 1964-1965, Mathematics Section, 158-
170, Moskov. Energet. Inst., Moscow 1965
\bibitem{PS}Psaradakis, G.; Spector, D. :
\emph{A Leray-Trudinger inequality.} J. Funct. Anal. 269 (2015), no. 1, 215-228. 
\bibitem{R} Ruf, B. : 
\emph{A sharp Trudinger-Moser type inequality for unbounded domains in $\mathbb{R}^2$.} J. Funct. Anal. 219,
no. 2, pp. 340-367 (2005).
\bibitem{T1} Tintarev, C. : 
\emph{Trudinger-Moser inequality with remainder terms.} J. Funct. Anal. 266, no. 1, 55-66
(2014).
\bibitem{T} Trudinger, Neil S. :
\emph{On imbeddings into Orlicz spaces and some applications.} 
J. Math. Mech. 17 1967 473-483.
\bibitem{VZ} Vazquez, Juan L.; Zuazua, E. :
\emph{The Hardy inequality and the asymptotic behaviour of the heat equation with an inverse-square potential.} J. Funct. Anal. 173 (2000), no. 1, 103-153.
\bibitem{GD} Wang, G. ;  Ye, D.  : 
\emph{A Hardy-Moser-Trudinger inequality.} Adv. Math. 230 (2012), no. 1, 294-320. 
\bibitem{Yud}Yudovich, V. I., Some estimates connected with integral operators and
with solutions of elliptic equations, Sov. Math., Dokl. 2 (1961), 746-749.
\end{thebibliography}
\end{document}